\documentclass{amsart}

\title{Weyl group invariants}
\author{Masaki Kameko}
\address{
Department of Mathematical Sciences,
Shibaura Institute of Technology,
307 Minuma-ku Fukasaku, Saitama-City 337-8570, Japan}
 \email{kameko@shibaura-it.ac.jp}
\thanks{The first named author is partially supported by the Japan Society for the Promotion of Science, Grant-in-
Aid for Scientific Research (C) 22540102.}
\author{Mamoru Mimura}
\address{
Department of Mathematics,
Faculty of Science,
Okayama University,
3-1-1 Tsushima-naka, Okayama, 700-8530,
Japan 
}
 \email{mimura@math.okayama-u.ac.jp}

\subjclass{55R40}

\newtheorem{theorem}{Theorem}[section]
\newtheorem{proposition}{Proposition}[section]
\newtheorem{lemma}{Lemma}[section]

\newtheorem{conjecture}{Conjecture}[section]

\theoremstyle{definition}

\newtheorem{remark}{Remark}[section]

\makeatletter
\let\c@proposition=\c@theorem
\let\c@lemma=\c@theorem
\let\c@corollary=\c@theorem
\let\c@conjecture=\c@theorem
\let\c@definition=\c@theorem
\let\c@remark=\c@theorem
\let\c@assumption=\c@theorem
\makeatother

\usepackage[usenames]{color}

\newcommand{\spin}{\mathrm{Spin}}

\begin{document}

\begin{abstract}
For any odd prime $p$, we prove that the induced homomorphism from the mod $p$ cohomology of the  
classifying space of a 
compact simply-connected simple connected Lie group to the Weyl group invariants of  the mod $p$
 cohomology of  the classifying space of its maximal torus is an epimorphism except for the case $p=3$, $G=E_8$.
\end{abstract}

\maketitle

\section{Introduction}

Let $p$ be an odd prime.
Let $G$ be a compact connected Lie group.
Let $T$ be a maximal torus of $G$. We denote by $W$ the Weyl group $N_G(T)/T$ of $G$.
We write $H^{*}(X)$ for the mod $p$ cohomology of a space $X$. 
Then, the Weyl group $W$ acts on $G$, $T$, $G/T$, $BG$, $BT$ and their cohomologies through the inner automorphism.
The mod $p$ cohomology of $BT$ is a polynomial algebra $\mathbb{Z}/p[t_1, \dots, t_n]$. We denote by $H^{*}(BT)^W$ the ring of invariants of the Weyl group $W$. Since $G$ is path connected, the action of the Weyl group on $BG$ is homotopically trivial and so the action of the Weyl group on the mod $p$ cohomology $H^{*}(BG)$ is trivial. Therefore, we have the induced homomorphism 
\[
\eta^*: H^{*}(BG)\to H^{*}(BT)^W.
\]

If $H_*(G;\mathbb{Z})$ has no $p$-torsion, the induced homomorphism is an isomorphism.
Even if $H_*(G;\mathbb{Z})$ has  $p$-torsion, the fact that  the induced homomorphism $\eta^*$  is an
epimorphism  was proved by Toda in \cite{toda-1973}  for $(G, p)=(F_4, 3)$ and announced in \cite{toda-1975} for $(G,p) =
(E_6,3)$,  respectively. However the results depend on the computation of the Weyl group
invariants.
The purpose of this paper is not only to show the following theorem but also to give a proof without explicit computation of the Weyl group invariants.



We denote by $y_2$ a generator of $H^{2}(BG)$ for $(G, p)=(PU(p), p)$ and by $y_4$ a generator of $H^{4}(BG)$ for $(G, p)=(F_4, 3)$, $ (E_6, 3)$, $(E_7, 3)$, $(E_8, 5)$. Let $Q_i$ be the Milnor operation 
defined by $Q_0=\beta$, $Q_1=\wp^1\beta-\beta \wp^1$, $Q_2=\wp^p Q_1-Q_1\wp^p$, \dots, where $\wp^i$ is the $i$-th Steenrod reduced power operation. Let $e_2=Q_0Q_1y_2$, $e_3=Q_1Q_2y_4$ for the above $H^{*}(BG)$'s. 
For a graded vector space $M$, we denote by $M^{even}$, $M^{odd}$ for graded subspaces of $M$ spanned by even degree elements and 
odd degree elements, respectively.

\begin{theorem}\label{main}
For $(G, p)=(PU(p),p), (F_4, 3), (E_6, 3), (E_7, 3), (E_8, 5)$, 
the induced homomorphism $\eta^*$ above
is an epimorphism. Moreover, we have \[
H^{*}(BT)^W=H^{even}(BG)/(e_k),\]
where $k=2$ for $(G, p)=(PU(p),p)$ and $k=3$ for $(G, p)=(F_4, 3), (E_6, 3), (E_7, 3), (E_8, 5)$.
\end{theorem}



If  $G$ is a simply-connected, simple, compact connected Lie group, 
then $G$ is one of the classical groups $SU(n)$, $Sp(n)$ and $\spin(n)$
or one of the exceptional groups $G_2$, $F_4$, $E_6$, $E_7$, $E_8$.
Since $H_*(G;\mathbb{Z})$ has no $p$-torsion except for the cases
$(G, p)=(F_4, 3)$, $(E_6, 3)$, $(E_7, 3)$, $(E_8, 3)$ and $(E_8, 5)$,  the above theorem 
provides a supporting evidence  for the following conjecture. 

\begin{conjecture}
Let $p$ be an odd prime.
Let $G$ be a simply-connected, simple, compact connected Lie group.
Then, the induced homomorphism $\eta^*$ above
is an epimorphism.
\end{conjecture}

To prove this conjecture, it remains to prove the case $(G, p)=(E_8, 3)$. However, the mod $3$ cohomology of $BE_8$ seems to be rather different from the other cases. For instance, the Rothenberg-Steenrod spectral sequence for the mod $p$ cohomology for $(G, p)$'s in Theorem~\ref{main} collapses at the $E_2$-level but 
the one for the mod $3$ cohomology of $BE_8$ is known not to collapse at the $E_2$-level and 
its computation is still an open problem. See \cite{kameko-mimura-2007-e8-3}.

In this paper, for $(G, p)$'s in Theorem~\ref{main}, we examine the Leray-Serre spectral sequence associated with 
the fibre bundle $BT \to BG$ and show that at the level of the spectral sequence, 
$E_r^{*,0}\not =(E_{r}^{*, *'})^{W}$ for some $r$ but we have $E_{\infty}^{*,0}=(E_{\infty}^{*, *'})^{W}$
at the end.
In the case $(G, p)=(E_8, 3)$, the cohomology of the base space $BE_8$ is not yet known;
 since we need the cohomology of the base space in order to examine the Leary-Serre spectral sequence, we
do not deal with the case $(G,p)=(E_8,3)$ in this paper.

The paper is organized as follows.
In \S2, we recall the Leray-Serre spectral sequence and the action of the Weyl group on it.
In \S3, we recall  the invariant theory  of the Weyl group of non-toral elementary abelian $p$-subgroup of $G$ in order to describe the cohomology of the base space $BG$.
In \S4, we recall the cohomology of $BG$ and prove Proposition~\ref{image}.
In \S5, using Proposition~\ref{image},  we compute  the Leray-Serre spectral sequence. As a consequence of the computation of  the Leray-Serre spectral sequence,  we obtain Theorem~\ref{main}.

The result in this paper is announced in \cite{kameko-mimura-2012}, where we deal with the case $(G, p)=(PU(p), p)$ in detail.


\section{The Weyl group and the spectral sequence}

As in \S1, let $G$ be a compact connected Lie group. We consider the
 Leray-Serre spectral sequence associated with the fibre bundle
\[
G/T \stackrel{\iota}{\longrightarrow} BT\stackrel{\eta}{\longrightarrow} BG.
\]
Since $BG$ is simply connected, the $E_2$-term is given by
\[
H^{*}(BG) \otimes H^{*'}(G/T).
\]
It converges to $gr\, H^{*}(BT)$.  Moreover, the Weyl group acts on this spectral sequence and its action is given by
\[
r^*(y\otimes x)=y \otimes r^* x, 
\]
where $r$ is an element in $W$. We fix a set of generators $\{ r_j\}$ and denote by $\sigma_j$
the induced homomorphism $1-r_j^*$. 
It is clear that \[
H^*(G/T)^W=\bigcap_{j} \; \mathrm{Ker}\, \sigma_j,
\]
and  $\sigma_j(x\otimes y)=x \otimes \sigma_j(y)$. Moreover, we have
\[
(E_r^{*,*'})^{W}=\bigcap_{j} \; \mathrm{Ker}\, \sigma_j.
\]

To relate the Weyl group invariants of $H^{*}(BT)$ with the one of $E_\infty$-term, that is $gr\, H^*(BT)$, of the spectral sequence, we use the following lemma.



\begin{lemma}\label{lemma-2-1}
Suppose that $f:M\to N$ is a filtration preserving homomorphism of finite dimensional vector spaces with filtration.
Denote by $gr f:gr M\to gr N$ the induced homomorphism between associated graded vector spaces.
Then, we have 
\[
\dim \mathrm{Ker}\; gr f \;  \geq \dim \mathrm{Ker}\; f.
\]
\end{lemma}

It is clear that 
\[
E_\infty^{*, 0} = \mathrm{Im}\; \eta^*:H^{*}(BG)\to H^{*}(BT)^W,
\]
so that $\dim E_{\infty}^{*, 0} \leq \dim H^{*}(BT)^W$.
By Lemma~\ref{lemma-2-1} above, we have
\[
\sum_{*'}  \dim (E_\infty^{*-*',*'})^{W}\geq \dim H^{*}(BT)^W.
\]
Hence, if we have 
\[
(E_\infty^{*,*'})^W=E_\infty^{*,0}, 
\]
we obtain 
\[
\dim H^{*}(BT)^W \leq \dim E_{\infty}^{*, 0} 
\]
and the desired result $E_{\infty}^{*, 0} =  H^{*}(BT)^W$.



In \cite{kac-1984}, Kac mentioned the following theorem and Kitchloo gave the detailed account of it  in \S5 of \cite{kitchloo-2008}. 

\begin{theorem}[Kac, Kitchloo]\label{Kac}
Let $p$ be an odd prime. Let $G$ be a compact connected Lie group.
Let $T$ be a maximal torus of $G$ and $W$ the Weyl group of $G$. Then, we have
$H^{*}(G/T)^W=H^0(G/T)=\mathbb{Z}/p$.
\end{theorem}

Theorem~\ref{Kac} is the starting point of this paper. By Theorem~\ref{Kac}, we have 
\[
(E_2^{*,*'})^{W}=(H^{*}(BG)\otimes H^{*'}(G/T))^W=(H^{*}(BG)\otimes \mathbb{Z}/p)=E_2^{*,0}.
\]
Recall that  the cohomology $H^{*}(G/T)$  has no odd degree generators. So,
if $H_*(G;\mathbb{Z})$ has no $p$-torsion,   then the  $E_2$-term has no odd degree generators. Hence, it collapses at the $E_2$-level.
Thus, we have that
\[
(E_\infty^{*,*'})^W=E_{\infty}^{*,0}=H^{*}(BG).\]
Therefore, it is clear that the induced homomorphism $\eta^*:H^{*}(BG)\to H^{*}(BT)^W$is an isomorphism if $H_*(G;\mathbb{Z})$ has no $p$-torsion.

However, for $(G, p)$ in Theorem~\ref{main}, $H_*(G;\mathbb{Z})$ has $p$-torsion and we have odd degree generators in the $E_2$-level. 
These odd degree generators do not survive to the $E_{\infty}$-level.
So, the spectral sequence does not collapse at the $E_2$-level. We deal with the spectral sequence for $(G,p)$ in Theorem~\ref{main} in \S4 and we will see that $(E_r^{*, *'})^W \not = E_r^{*, 0}$ for some $r$ but still $(E_\infty^{*, *'})^W = E_\infty^{*, 0}$ holds.

We end this section by recalling from \cite{kac-1984} the description of the mod $p$ cohomology of $G/T$ for $(G, p)$'s in Theorem~\ref{main}.



\begin{theorem}[Kac]\label{flag}
For $(G,p)$ in Theorem~{\rm\ref{main}}, as an $S$-module, 
$H^{*}(G/T)$ is a free $S$-module generated by $x_m^i$ $(0\leq i \leq p-1)$, that is, 
\begin{align*}
H^{*}(G/T) &= S\{ x_m^{i} \;|\; 0\leq i \leq p-1 \},
\end{align*}
where $S$ is the image of the induced homomorphism $\iota^*:H^{*}(BT) \to H^{*}(G/T)$, $m=2$ for $(G, p)=(PU(p), p)$ and $m=2p+2$ for $(G,p)=(F_4, 3)$, $( E_6, 3)$, $(E_7, 3)$ and $(E_8, 5)$.
\end{theorem}


\section{Invariant theory}

In order to describe the odd degree generators of $H^{*}(BG)$ for $(G, p)$ in Theorem~\ref{main}, 
we consider non-toral elementary abelian $p$-subgroups of $G$.

Non-toral elementary abelian $p$-subgroups  of a compact connected Lie group $G$ and their Weyl groups  are described in \cite{agmv} not only for $(G, p)$ in Theorem~\ref{main} but also for $(G, p)=(E_8, 3), (PU(p^n), p)$. 
For $(G, p)$'s in Theorem~\ref{main}, 
there exists a unique, up to conjugacy,  maximal non-toral elementary abelian $p$-subgroup $A$. 
Their Weyl groups
$W(A)=N_G(A)/C_G(A)$ are also determined in  \cite{agmv}. We refer the reader to  \cite{agmv} for the details.
From now on, we consider the case $(G, p)$ in Theorem~\ref{main} only.

We denote by $\xi:A \to G$ the inclusion of $A$ into $G$ and  the induced map $BA\to BG$ by the same symbol $\xi:BA \to BG$.
Although the Weyl group invariants $H^*(BT)^{W}$ is not yet known, we determined in \cite{kameko-mimura-2007} 
the ring of invariants $H^{*}(BA)^{W(A)}$. It  is rather easy to describe it in terms of Dickson-Mui invariants because the Weyl groups $W(A)$ are $SL_{2}(\mathbb{Z}/p)$ for $(G, p)=(PU(p), p)$, 
$SL_3(\mathbb{Z}/p)$ for $(G, p)=(F_4, 3)$, $(E_8, 5)$ and 
\[ \left\{ \left( \begin{array}{c|ccc}
1 & * &* & * \\
\hline
0 &   &  &  \\
0 & & g & \\
0 & & & 
\end{array}
\right) \right\}, \quad  \left\{ \left( \begin{array}{c|ccc}
\varepsilon & * &* & * \\
\hline
0 &   &  &  \\
0 & & g & \\
0 & & & 
\end{array}
\right) \right\}
\]
for $(G, p)=(E_6, 3)$, $(E_7, 3)$, respectively, where $g$ ranges over  $SL_3(\mathbb{Z}/3)$ and $\varepsilon$ ranges over $(\mathbb{Z}/3)^{\times}$.



In order to  describe  the image of $\xi^*:H^{*}(BG)\to H^{*}(BA)^{W(A)}$ for the above $(G, p)$, 
firstly, we recall from \cite{kameko-mimura-2007} the invariant theory of special linear groups and related groups.
Let $p$ be an odd prime and 
 $A_n$  the elementary abelian $p$-group of rank $n$.  
 We need the cases $n=2, 3, 4$, $p=3, 5$ only. 
 However, since some arguments seem to be comprehensive 
 when we put them in the general  setting, 
 we recall the invariant theory without any restriction on $p$ and $n$ for the time being. 
 We have
\[
H^{*}(BA_n)=\mathbb{Z}/p[t_1, \dots, t_n] \otimes \Lambda(dt_1, \dots, dt_n), 
\]
where $dt_i$'s are generators of  $H^1(BA_n)$, $t_i=\beta dt_i$, and $\beta$ is the Bockstein homomorphism.
Denote by
$G_n$, $G_n'$   subgroups of $GL_n(\mathbb{Z}/p)$ consisting of the following matrices:
\[
\left( \begin{array}{c|ccc} 1 & * & \cdots  & * \\ \hline
0 & & &  \\[-1ex]
\vdots & & g &  \\
0 & & &  
\end{array}
\right), \quad \left( \begin{array}{c|ccc} \alpha & * & \cdots & * \\ \hline
0 & & &  \\[-1ex]
\vdots & & g &  \\
0 & & & 
\end{array}
\right), 
\]
respectively, where $g$ ranges over $SL_{n-1}(\mathbb{Z}/p)$ and  $\alpha$ ranges over $ (\mathbb{Z}/p)^{\times}$.
We use the letter $H$ to denote one of  $SL_n(\mathbb{Z}/p), G_n, G_n'$.

Now, we define some elements in the ring of  invariants.
Let $V_n$ 
be the vector space over $\mathbb{Z}/p$ spanned by 
$t_1, \dots, t_n$ and $V_{n-1}$ the subspace spanned by $t_2, \dots, t_n$.
We denote the element $dt_1\cdots dt_n$ by $u_n$ and let $e_n=Q_0\cdots Q_{n-1} u_n$.
We also denote $dt_2\cdots dt_{n}$ by $u_{n-1}$ and 
let $e_{n-1}=Q_0\cdots Q_{n-2} u_{n-1}$.
The Dickson invariants $c_{n, i}$'s are defined by
\[
\prod_{x\in V_n}  (X-x)=\sum_{j=0}^{n} (-1)^{j}c_{n, n- j} X^{p^{n-j}}.
\]
 We also consider the Dickson invariants  $c_{n-1, i}$ given by
\[
\prod_{x\in V_{n-1}}  (X-x)=\sum_{j=0}^{n-1} (-1)^{j}c_{n-1, n- 1-j} X^{p^{n-1-j}}.
\]
We define $\mathcal{O}_{n-1}$ to be \[
Q_{n-1}-c_{n-1, n-2} Q_{n-2}+\cdots +(-1)^{n-1}c_{n-1, 0} Q_0.
\]
Let $f_n=\mathcal{O}_{n-1}(dt_1)$.
Then, it is clear that we have
\[
u_{n-1}=f_n^{-1}  \mathcal{O}_{n-1} u_n
\]
and 
\[
e_{n-1}=f_{n}^{-1} Q_0\cdots Q_{n-2}\mathcal{O}_{n-1} u_n.
\]
For the sake of notational simplicity, we denote $\{ 0, \dots, n-1\}$, $\{ 0, \dots, n-2\}$ by $\Delta_n$, $\Delta_{n-1}$, respectively.
For a  subset $I=\{ i_1, \dots, i_r\}$, where
$0\leq i_1<\cdots <i_r\leq n-1$,  of $\Delta_n$, 
let
$$Q_I u_n=Q_{i_1}\cdots Q_{i_r} u_n$$ and $Q_{\phi}u_n=u_n$. 
In order to deal with $H=SL_n(\mathbb{Z}/p), G_n, G'_n$ in the same manner, we need  the following definition.
We define $\overline{Q}_i$ by $\overline{Q}_i=Q_i$ for $i=0, \dots, n-2$.
Let $\overline{u}_n=u_n$ for $H=SL_n(\mathbb{Z}/p), G_n$ and 
$\overline{u}_n=f_n^{p-2} u_n$ for $H=G'_n$.
We define $\overline{Q}_{n-1}$ by
 $$\overline{Q}_{n-1}=Q_{n-1}, \;\; f_n^{-1}{\mathcal{O}}_{n-1}, \;\;  f_n^{-p+1}\mathcal{O}_{n-1}$$ for $H=SL_n(\mathbb{Z}/p), G_n, G_n'$, respectively.
We define $\overline{Q}_I \overline{u}_n$ to be
$$\overline{Q}_{i_1}\cdots \overline{Q}_{i_{r}} \overline{u}_n$$
and $\overline{Q}_{\emptyset}\overline{u}_n=\overline{u}_n$.

We denote the ring of invariants 
\[
\mathbb{Z}/p[t_1, \dots, t_n]^{H}
\]
by $R$.
We have 
\[
\begin{array}{rcll} 
R&=&\mathbb{Z}/p[c_{n,1}, \dots, c_{n,n-1}, e_n] & \quad \mbox{for $H=SL_n(\mathbb{Z}/p)$,}\\
R&=&\mathbb{Z}/p[c_{n-1, 1}, \dots, c_{n-1,n-2}, e_{n-1}, f_n] &\quad   \mbox{for $H=G_n$,}\\
R&=&\mathbb{Z}/p[c_{n-1, 1}, \dots, c_{n-1,n-2}, e_{n-1}, f_n^{p-1}] & \quad  \mbox{for $H=G'_n$.}
\end{array}
\]
Moreover, we have 
\[
\begin{array}{rcll}
H^{*}(BA_n)^{H}&=&
R \{ 1, 
Q_I u_n  \}& \quad \mbox{for $H=SL_n(\mathbb{Z}/p)$,} \\
H^{*}(BA_n)^{H}&=&R\{ 1, Q_J u_{n-1}, Q_K u_n \}&\quad \mbox{for $H=G_n$,} \\
H^{*}(BA_n)^{H}&=&R\{ 1, Q_J u_{n-1}, f_n^{p-2}Q_K u_n \}&\quad \mbox{for $H=G_n'$,} 
\end{array}
\]
where  $I$ ranges over all the proper subsets of $\Delta_{n}$, $J$ ranges over all the proper subsets of $\Delta_{n-1}$
and $K$ ranges over all the subsets of $\Delta_{n-1}$.
With the above definitions of $\overline{Q}_i$'s and $\overline{u}_n$, 
 we may write
\[
H^{*}(BA_n)^{H}=R \{ 1, \overline{Q}_I \overline{u}_n\}, 
\]
where $I$ ranges over all the proper subsets of $\Delta_{n}$.



Secondly, we consider a subspace $F_i$ of $H^{*}(BA_n)^{H}$.
Since there holds 
 \[
  H^{*}(BA_n)^{G_n} \subset H^{*}(BA_n)^{G_n},
 \]
we introduce a  filtration on $H^{*}(BA_n)^{SL_n(\mathbb{Z}/p)}$ and $H^{*}(BA_n)^{G_n}$.
 Let $w(\overline{Q}_I\overline{u}_n)$ be the number of elements in $I$.
 For monomials in $H^{*}(BA_n)^{SL_n(\mathbb{Z}/p)}$, 
 we define $w(-)$ by
 \begin{align*}
 w(e_n)&=n, \\
 w(c_{n, j})&=0 \quad \mbox{($j=1, \dots, n-1$)}, 
 \end{align*}
 and $$w(xy)=w(x)+w(y),$$
 where $x, y$  are  monomials in $H^{*}(BA_n)^{SL_n(\mathbb{Z}/p)}$.
For monomials  in $H^{*}(BA_n)^{G_n}$, we define $w(-)$ by
\begin{align*}
w(e_{n-1})&=n, \\
w(c_{n-1, j})&=0 \quad \mbox{($j=1, \dots, n-2$)}, \\
w(f_n)&=0,
\end{align*}
and $$w(xy)=w(x)+w(y),$$
where $x, y$  are  monomials in $H^{*}(BA_n)^{G_n}$.
We  denote by $F_{j+i, i}$ the subspace spanned by elements $x$ such that
$j+i \geq w(x)\geq i$, that is, $F_{j+i, i}=F_i\oplus F_{i+1} \oplus \cdots \oplus F_{i+j}$.
 Since  
 \[
 H^{*}(BA_n)^{G_n'} \subset  H^{*}(BA_n)^{G_n},
 \]
 by  abuse of notation, we denote 
 $$H^{*}(BA_n)^{G_n'} \cap F_{i}$$ by $F_{i}$.
It is clear that 
\[
F_{\infty, i}= \mathbb{Z}/p[e_{k}] \otimes F_{n-1+i, i}, 
\]
where $k=n$ for $H=SL_n(\mathbb{Z}/p)$ and $k=n-1$ for $H=G_n, G'_n$.
It is also clear that
\[
H^{*}(BA_n)^H=R \oplus F_{\infty,  0}.
\]

We will see in Theorem~\ref{detect} that 
$R\oplus F_{\infty, 0}$ is the image of $\xi^*$ for $(G, p)=(PU(p), p)$, 
that $R\oplus F_{\infty, 1}$ is the image of $\xi^*$ for $(G, p)=(F_4, 3)$ and $(E_8, 5)$, 
that $R\oplus F_{\infty,2}$ is the image of $\xi^*$ for $(G, p)=(E_6, 3)$ and that
$R\oplus (F_{\infty, 2}  \cap (\overline{\mathcal{O}}_3 u_4) )$ is the image of $\xi\*$ for $(G, p)=(E_7, 3)$ where $ (\overline{\mathcal{O}}_3 u_4) $ is the $R$-submodule generated by 
$Q_{I}u_3=Q_I\overline{\mathcal{O}}_3 u_4$'s.



Next, we consider the multiplication on $R \oplus F_{\infty, i}$.
For $I=\{i_1, \dots, i_r\}, J=\{j_1, \dots, j_s\} \subset \Delta_n$
such that  $I \cap J=\emptyset$, we denote the sign of the permutation
\[ 
\left( \begin{array}{cccccc}
1 & \cdots & r & r+1 & \cdots & r+s\\ i_1 & \cdots & i_r & j_1 & \cdots & j_s \end{array}\right)
\]
simply by $\mathrm{sgn}(I, J)$. If one of $I$, $J$ is empty, we set $\mathrm{sgn}(I, J)=1$.

\begin{lemma}\label{lemma-3-1}
If $I \cup J \not =\Delta_{n}$, then $Q_Iu_n\cdot Q_Ju_n=0$.
If $I\cup J=\Delta_{n}$, then 
\[
Q_I u_n \cdot Q_Ju_n=(-1)^{nr+r^2}\mathrm{sgn}(K, I  \setminus K)  \mathrm{sgn}(I\setminus K, J) e_n Q_K u_n,
\]
where $K=I \cap J$ and $r$ is the number of elements in $I\setminus K$.
\end{lemma}

\begin{proof}
If $I \cup J \not = \Delta_n$, then 
$w(Q_{I\setminus K} u_n)+w(Q_Ju_n)<n$. So, it is clear that 
\[
Q_{I\setminus K} u_n \cdot Q_J u_n=0.
\]
On the other hand, we have
\[
Q_K (Q_{I\setminus K} u_n \cdot Q_J u_n)=\mathrm{sgn}(K, I \setminus K) Q_I u_n \cdot Q_Ju_n, 
\]
since $Q_k Q_J=0$ for $k \in K \subset J$. Hence, we have $Q_I u_n \cdot Q_Ju_n=0$ as desired.

Next, we deal with the case $I \cup J=\Delta_{n}$ and $I\cap J=\emptyset$. 
There holds
\[
Q_{I\setminus \{i_1\}} u_n \cdot Q_Ju_n=0, 
\]
since $w(Q_{I\setminus \{i_1\}} u_n)+w(Q_Ju_n)<n$.
On the other hand, we have
\begin{align*}
Q_{i_1} (Q_{I\setminus \{i_1\}} u_n \cdot Q_Ju_n)&=Q_I u_n \cdot Q_J u_n +(-1)^{n+r-1} Q_{I\setminus \{i_1\}} u_n \cdot Q_{i_1} Q_{J} u_n.
\end{align*}
Hence, we have 
\begin{align*}
Q_I u_n \cdot Q_J u_n &=(-1)^{n+r} Q_{I\setminus \{i_1\}} u_n \cdot Q_{i_1} Q_{J} u_n.
\end{align*}
Therefore, we obtain 
\begin{align*}
Q_I u_n \cdot Q_Ju_n&=(-1)^{n+r}(-1)^{n+r-1}\cdots (-1)^{n+1}   u_n \cdot Q_{i_r} \cdots Q_{i_2} Q_{i_1} Q_J u_n
\\
&=(-1)^{nr+r(r+1)/2} (-1)^{r(r-1)/2} u_n \cdot Q_I Q_J u_n \\
&=(-1)^{rn+r^2}\mathrm{sgn}(I, J) u_n \cdot e_n,
\end{align*}
which is the desired result.

Now, we deal with the  case $I \cup J=\Delta_n$ and $K=I \cap J$. 
Then, we have 
\[
Q_{I\setminus K} u_n \cdot Q_J u_n=(-1)^{rn+r^2}\mathrm{sgn}(I\setminus K, J) e_n u_n.
\]
Hence, we have 
\[
Q_K(Q_{I\setminus K} u_n \cdot Q_J u_n)=(-1)^{rn+r^2} \mathrm{sgn}(I\setminus K, J) e_n Q_K u_n.
\]
On the other hand, we have
\[
Q_K (Q_{I\setminus K} u_n \cdot Q_J u_n)=\mathrm{sgn}(K, I\setminus K) Q_I u_n  \cdot Q_Ju_n,
\]
since $Q_k Q_J u_n=0$ for $k \in K \subset J$.
This completes the proof.
\end{proof}

\begin{remark} \label{remark-3-2} For $H=G_n, G_n'$, 
it is easy to see that
we have the same formula
\begin{align*}
\overline{Q}_I   \overline{u}_n \cdot \overline{Q}_J  \overline{u}_n&
=(-1)^{nr +r^2} 
\mathrm{sgn}(K , I \setminus K)
\mathrm{sgn}(I\setminus K, J ) e_{n-1} \overline{Q}_K \overline{u}_n.
\end{align*}
\end{remark}

By Lemma~\ref{lemma-3-1} and by Remark~\ref{remark-3-2}, we see that $F_i \cdot F_j \subset F_{i+j}$.
Therefore, $R\oplus F_{\infty,  i}$ is closed under the multiplication.

Finally, we end this section by considering the direct sum decomposition $$R \oplus F_{\infty, i}=N_0\oplus N_1.$$
We use this decomposition in order  to deal with  differentials in the spectral sequence in \S5.
For $0\leq \ell \leq n-2$, 
let $E{(\ell)}$ be the subspace of $F_{\infty,  0}$ spanned by $\{ x\overline{ Q}_I \overline{u}_n \; |\; \ell \in I, x\in R \}$ and
$\widehat{E}{(\ell)}$ the subspace of $F_{\infty,  0}$ spanned by $\{ x\overline{ Q}_I \overline{u}_n \; |\; \ell \not \in I, x \in R\}$.
Let
\begin{align*}
N_0&=R\oplus (F_{\infty,  i} \cap E(\ell))\quad \quad\mbox{and} \\
N_1 &=F_{\infty,  i} \cap \widehat{E}(\ell).
\end{align*}
Let 
\[
z_\ell  =Q_0\cdots \widehat{Q}_\ell \cdots Q_{n-2} \overline{Q}_{n-1} \overline{u}_n.
\]
From now on, we assume that $i\leq n-1$, so that $z_\ell \in N_1$.
Let $\overline{F}_{j,i}={F}_{j,i}/(z_{\ell})$ if $z_\ell \in F_{\infty,  i}$.
Then, 
it is easy to see that the following proposition holds.



\begin{proposition}  \label{prop-3-3}
There hold the following\,{\rm :}
\begin{itemize}
\item[(1)] Suppose that  $i<n$. 
If  $n-i$ is even, then $F_i^{even}=F_i$. If not, $F_i^{even}=\{0\}$.
\item[(2)] If $n$ is even, then $F_n^{even}=F_n$. If not, $F_n^{even}=\{0\}$.
\item[(3)] $\overline{F}_{n-1} \cap \widehat{E}(\ell)=\{0\}$.
\item[(4)] $\overline{F}_n \cap \widehat{E}(\ell)=F_n$.
\end{itemize}
\end{proposition}

The Milnor operation $Q_\ell$ induces a short exact sequence
\[
0 \to N_1 \stackrel{Q_\ell}{\longrightarrow}  N_0 \to R/(e_k)\oplus {F}_i \cap E(\ell) \to 0. \]
The multiplication by $z_\ell$ induces an isomorphism 
\[
0\to N_0\stackrel{z_\ell}{\longrightarrow} N_1 \to  \overline{F}_{n-2+i, i} \cap \widehat{E}(\ell) \to 0.
\]
We have the following proposition:



\begin{proposition}\label{prop-3-4}
For $(n, i)=(2, 0), (3, 1), (4, 2)$, 
there exist short exact sequences
\begin{itemize}
\item[(1)] \quad $\displaystyle 0\to N_0 \stackrel{z_{\ell}}{\longrightarrow}  N_1 \to N_1^{even}/(e_k)\to 0$,
\item[(2)] \quad $\displaystyle 0 \to N_1\stackrel{Q_{\ell}}{\longrightarrow}  N_0 \to N_0^{even}/(e_k)  \to 0$.
\end{itemize}
\end{proposition}

\begin{proof}
From the observation above, it suffices to show that 
\begin{align*}
R/(e_k) \oplus F_i \cap E(\ell) & = N_1^{even}/(e_k) \quad \mbox{and} \\
\overline{F}_{n-2+i, i} \cap \widehat{E}(\ell) &= N_0^{even}/(e_k).
\end{align*}
Since $(N_0\oplus N_1)^{even}/(e_k)=R/(e_k) \oplus F^{even}_{n-1+i, i}$, it suffices to show that 
\[
F_{n-1+i, i}^{even}=F_i \cap E(\ell) \oplus \overline{F}_{n-2+i, i}\cap \widehat{E}(\ell).
\]
By Proposition~\ref{prop-3-3}, we have the following table.
\[
 \renewcommand{\arraystretch}{1.3}
\begin{array}{c|c|c|c}
(n, i) & (2, 0) & (3, 1) & (4, 2) \\
\hline
F_i \cap E(\ell) & F_0^{even} \cap E(\ell) & F_1^{even} \cap E(\ell) & F_{2}^{even}\cap E(\ell) \\
\hline
\overline{F}_{i} \cap \widehat{E}(\ell) & F_0^{even} \cap \widehat{E}(\ell) & F_1^{even} \cap \widehat{E}(\ell) & F_2^{even} \cap \widehat{E}(\ell) \\
\overline{F}_{i+1} \cap \widehat{E}(\ell) & - &  0 & 0 \\
\overline{F}_{i+2} \cap \widehat{E}(\ell) & - & - &  F_4^{even}\\
\hline
F_{i}^{even} & F_0^{even} &  F_1^{even} &  F_2^{even} \\
F_{i+1}^{even} &0 & 0 & 0  \\
F_{i+2}^{even} & - & 0 & F_4^{even} \\
F_{i+3}^{even} & - & - & 0 
\end{array}
\]
This table completes the proof.
\end{proof}

\section{Cohomology of  classifying spaces}



Now, we investigate $H^{*}(BG)$ and $H^{*}(BA)^{W(A)}$ for $(G, p)$ in Theorem~\ref{main}.
Throughout the rest of  this section, 
we put  
\vspace{.5ex}
\\
{
 \renewcommand{\arraystretch}{1.3}
\begin{tabular}{lllll}
$k=2$, & $n=2$, & $R=\mathbb{Z}/p[e_2, c_{2, 1}]$ & for & $(G, p)=(PU(p),p)$,\\
$k=3$,  & $n=3$, & $R=\mathbb{Z}/p[e_3, c_{3, 1}, c_{3,2}]$ & for & $(G, p)=(F_4, 3), (E_8, 5)$, \\
$k=3$, & $n=4$, & $R=\mathbb{Z}/p[e_3, c_{3, 1}, c_{3,2}, f_4]$ & for & $(G, p)=(E_6, 3)$ and \\
$k=3$, & $n=4$, & $R=\mathbb{Z}/p[e_3, c_{3, 1}, c_{3,2}, {f_4}^2]$ & for & $(G,p)=(E_7, 3)$.\\
\end{tabular}
}
\\

\noindent Then, the number $n$ is nothing but the rank of the maximal non-toral elementary $p$-subgroup $A$.
We recall from \cite{kameko-mimura-2007} the following.

\begin{theorem}\label{kameko-mimura}
The ring of invariants $H^{*}(BA)^{W(A)}$ is
\begin{align*}
R \{1,  Q_I u_2\}
& \quad \mbox{for $(G, p)=(PU(p), p)$, }
\end{align*}
where $I$ ranges over all the proper subsets of $\Delta_2$,
\begin{align*}
R\{ 1, Q_I u_3\} & \quad 
\mbox{for $(G, p)=(F_4, 3)$, $(E_8, 5)$,} \\
R \{ 
1, Q_I u_3, Q_J u_4 \}& \quad 
\mbox{for $(G, p)=(E_6, 3)$ and }
\\
R \{ 
1, Q_I u_3, f_4 Q_J u_4 \}
& \quad 
\mbox{for $(G, p)=(E_7, 3)$,}
\end{align*}
where $I$ ranges over all the proper subsets of $\Delta_3$ and $J$ ranges over all the subsets of $\Delta_3$.
\end{theorem}

For $(G, p)=(PU(p), p)$, the induced homomorphism $\xi^*$ defined in the previous section
 is an epimorphism but it is not an epimorphism for $(G, p)=(F_4, 3)$, $(E_6, 3)$, $(E_7, 3)$ and $(E_8, 5)$.



By comparing odd degree generators of $H^{*}(BG)$ with  the image of the induced homomorphism
$\xi^*$,  it is easy to see the following theorem. For $H^{*}(BG)$, we refer the reader to the work of Mimura and Sambe in \cite{mimura-sambe-1979} and \cite{mimura-sambe-1981}, 
or Kameko \cite{kameko-2010}.  

\begin{theorem}\label{detect}
The induced homomorphism 
\[
H^{odd}(BG) \to H^{odd}(BA)^{W(A)}
\]
is a monomorphism and 
the image of
\[
\xi^*: H^{*}(BG) \to H^{*}(BA)^{W(A)}
\]
 is generated by
$Q_0u_3$, $c_{3,1}$ and $c_{3,2}$   
for $(G, p)=(F_4, 3)$ and for $(E_8, 5)$,
$Q_0u_3$, $c_{3,1}$ and $c_{3,2}$, $\mathcal{O}_3(dt_1)^2$   
for $(G,p)=(E_7, 3)$,
and $Q_0u_3$, $c_{3,1}$, $c_{3,2}$, $Q_0Q_1u_4$, $\mathcal{O}_3(dt_1)$  for $(G, p)=(E_6, 3)$, 
as an algebra over the Steenrod algebra.
\end{theorem}


We choose elements $e_k$ in $H^{*}(BG)$ as in Theorem~\ref{main}. We choose elements corresponding to other generators of $R$ in $H^{*}(BG)$
and we consider $R$ as a subalgebra of $H^{*}(BG)$.

For $(G, p)=(PU(p), p)$, 
the rank of the elementary abelian $p$-subgroup $A$ is $2$. 
Since $F_{1, 0}=R/(e_2) \{ u_2, Q_0 u_2, Q_1 u_2\}$, we put
\begin{align*}
M_0 &= R\{ 1, y_{2p+1} \}, \\
M_1&=R\{ y_{3}, y_2\},
\end{align*}
where $y_2$ is the generator of $H^2(BG)$ such that $\xi^*(y_2)=u_2$ and $y_3=Q_0 y_2$, $y_{2p+1}=Q_1 y_2$.
Then, we have 
\begin{align*}
\xi^*M_0&=R \{ 1, Q_1 u_2\}=R\oplus (F_{\infty,  0}\cap {E}(1)), \\
\xi^*M_1&=R \{ Q_0u_2,  u_2\}=F_{\infty,  0} \cap \widehat{E}(1).
\end{align*}

For $(G, p)=(F_4, 3)$ and $(E_8, 5)$, the rank of the elementary abelian $p$-subgroup $A$ is $3$. 
Since $F_{3, 1}=R/(e_3) \{ Q_0 u_3, Q_1 u_3, Q_2 u_3, Q_0 Q_1 u_3, Q_0 Q_2 u_3, Q_1 Q_2 u_3, e_3 u_3 \}$, 
 we put
 \begin{align*}
 M_0 &= R\{ 1, y_{2p^2+2}, y_{2p^2+3}, y_{2p^2+2p+1}\}, \\
 M_1 & = R \{ y_{2p+3}, y_{2p+3}y_{2p^2+2}, y_4, y_{2p+2} \}, 
 \end{align*}
 where $y_4$ is the generator of $H^4(BG)$ such that $\xi^*(y_4)=Q_0u_3$, $y_{2p+2}=\wp^1 y_4$, $y_{2p+3}=-Q_1y_4$, 
$ y_{2p^2+2}=\wp^p\wp^1 y_4$, $y_{2p^2+3}=-Q_2y_4$, $y_{2p^2+2p+1}=-Q_2 y_{2p+2}$.
 Then, we have
\begin{align*}
\xi^*M_0&=R \{ 1, Q_2 u_3, Q_0Q_2 u_3, Q_1Q_2u_3\}=R\oplus (F_{\infty,  1} \cap {E}(2)), \\
\xi^*M_1&=R \{ Q_0Q_1u_3, e_3 u_3, Q_0u_3, Q_1u_3 \}=F_{\infty,  1}\cap \widehat{E}(2).
\end{align*}


For $(G, p)=(E_6, 3)$ and $(E_7, 3)$, the rank of the elementary abelian $p$-subgroup $A$ is $4$.
First, we consider the case $(G, p)=(E_6, 3)$.  Let $\overline{\mathcal{O}}_3=f_4^{-1} \mathcal{O}_3$.
Since 
\[
F_{5,2}=R/(e_3)\{\overline{Q}_I\overline{u}_4,  e_3 \overline{Q}_J\overline{u}_4\},
\]
where $w(\overline{Q}_I\overline{u}_4)=2, 3$ and $w(\overline{Q}_J\overline{u}_4)=0, 1$, 
we denote by  $y_4$, $y_{10}$  the generators of $H^4(BG)$, $H^{10}(BG)$  such that $\xi^*(y_4)=Q_0\overline{\mathcal{O}}_3 u_4=Q_0u_3$, $\xi^{*}(y_{10})=Q_0Q_1u_4$, respectively.
Let $y_9=-Q_1 y_4$,  $y_{8}=\wp^1 y_4$,  $y_{20}=\wp^3 \wp^1 y_4$ and let $y_{22}=\wp^3 y_{10}$, $y_{26}=\wp^1 y_{22}$.
Let $y_{30}=y_{10}y_{20}$.
Let
\begin{align*}
M_0&=R\{   1, y_{22}, y_{26}, y_{20}, Q_2 y_{10}, Q_2 y_4, Q_2 y_8,Q_2y_{30} \}, \\
M_1&= R\{  y_9,  y_9y_{22}, y_9y_{26}, y_9y_{20},  y_{10}, y_{4}, y_8, y_{30} \}.
\end{align*}
Then, we have 
\begin{align*}
\xi^*M_0&=R\{ 1, Q_iQ_2 u_4, Q_2 \overline{\mathcal{O}}_3 u_4, Q_2\overline{\mathcal{O}}_3u_4, 
Q_0Q_1Q_2 u_4,  
Q_iQ_2 \overline{\mathcal{O}}_3 u_4, 
e_3 Q_2u_4
\}\\
&=(R\oplus F_{\infty,  2})\cap {E}(2), \\
\xi^*M_1&=R\{ Q_0Q_1\overline{\mathcal{O}}_3u_4, e_3 \overline{\mathcal{O}}_3u_4, Q_i\overline{\mathcal{O}}_3u_4, Q_0Q_1u_4, e_3u_4, e_3Q_iu_4\}
\\
&=F_{\infty,  2}\cap \widehat{E}(2),
\end{align*}
where $i$ ranges over $\{0,1\}$.

As for $(G, p)=(E_7, 3)$, 
let 
\begin{align*}
M_0 &= R\{ 1, y_{20}, Q_2 y_4, Q_2 y_8 \}, \\
M_1 & = R\{ y_9, y_9 y_{20}, y_4, y_8\}.
\end{align*}
Then, we have 
\begin{align*}
\xi^*M_0&=R\{ 1, Q_2\overline{\mathcal{O}}_3u_4, Q_iQ_2\overline{\mathcal{O}}_3u_4,  \}=R\oplus (F_{\infty,  2}\cap (\overline{\mathcal{O}}_3 u_4) \cap {E}(2)), \\
\xi^*M_1&=R\{ Q_0Q_1\overline{\mathcal{O}}_3u_4, e_3 \overline{\mathcal{O}}_3u_4, Q_i\overline{\mathcal{O}}_3u_4\}=F_{\infty,  2}\cap (\overline{\mathcal{O}}_3 u_4) \cap \widehat{E}(2),
\end{align*}
where $i$ ranges over $\{ 0, 1\}$.

One can verify the following proposition by direct computations but it also follows immediately from 
Proposition~\ref{prop-3-4}.



\begin{proposition} \label{image} 
For $(G, p)$ in Theorem~{\rm \ref{main}}, there holds 
\begin{itemize}
\item[(1)] \quad $\xi^{*}M_0\oplus \xi^*M_1=\mathrm{Im}\, \xi^*$.
\end{itemize}
Moreover, there exist the following short exact sequences{\rm :}
\begin{itemize}
\item[(2)] \quad $\displaystyle 0 \to M_0 \stackrel{y_{m+1}}{\longrightarrow} M_1 \to M_1^{even}/(e_k) \to 0,
$
\item[(3)] \quad $\displaystyle  0 \to M_1\stackrel{Q_{k-1}}{\longrightarrow} M_0 \to M_0^{even}/(e_k) \to 0,
$
\end{itemize}
where $m=2$, $k=2$ for $(G, p)=(PU(p),p)$ and $m=2p+2$, $k=3$ for $(G,p)=(F_4,3), (E_6, 3), (E_7, 3), (E_8, 5)$.
\end{proposition}


\section{The spectral sequence}

In this section, we complete the proof of Theorem~\ref{main} by computing the Leray-Serre spectral sequence
associated with
\[
G/T \stackrel{\iota}{\longrightarrow} BT\stackrel{\eta}{\longrightarrow}
BG
\]
for $(G, p)$ in Theorem~\ref{main}.
We put $m=2$, $k=2$ for $(G, p)=(PU(p),p)$ and $m=2p+2$, $k=3$ for $(G,p)=(F_4,3), (E_6, 3), (E_7, 3), (E_8, 5)$.

The $E_2$-term of the spectral sequence is given by 
\[
E_2=H^{*}(BG) \otimes H^{*'}(G/T) 
\]
as an $H^{*}(BG)\otimes S$-algebra. The algebra generator is $1\otimes x_m$. 
So, the first nontrivial differential is determined by $d_r(1\otimes x_m)$ for some $r\geq 2$.


\begin{proposition} \label{first-differential}
For $r< m+1$, $d_r=0$. The first nontrivial differential is $d_{m+1}$ and 
there holds $$d_{m+1}(1\otimes x_m)=\alpha (y_{m+1}\otimes 1)$$ for some $\alpha\not =0\in \mathbb{Z}/p$.
\end{proposition}
\begin{proof}
Suppose that $d_{r_0}(1\otimes x_m) \not =0$ for some $r_0<m+1$. Then, up to degree $\leq m$, $E_{r_0+1}$-term is generated by $1\otimes 1$ as an $H^{*}(BG)\otimes S$-module. So, for $r_1\geq r_0$, $\mathrm{Im}\, d_{r_1}$ does not contain any element of degree less than or equal to $m+1$. Hence, $y_{m+1}\otimes 1$ survives to the $E_\infty$-term. Then, $\eta^*(y_{m+1})\not=0$. This contradicts the fact $E_{\infty}^{odd}=\{0\}$, since $\deg y_{m+1}=m+1$ is odd.
Therefore, we have $d_r(1\otimes x_m)=0$ for $r<m+1$.

Next, we verify that $d_{m+1}(1\otimes x_m)=\alpha ( y_{m+1}\otimes 1) $ for some $\alpha \not=0$ in $\mathbb{Z}/p$.
If $\mathrm{Im}\, d_{m+1}$ does not contain $y_{m+1}\otimes 1$, then up to degree $\leq m+1$, the spectral sequence collapses at the $E_{m+2}$-level and $y_{m+1}\otimes 1$ survives to the $E_{\infty}$-term. As in the above, it is a contradiction. Hence, the proposition holds.
\end{proof}

To consider the next nontrivial differential, first we show the following lemmas.



\begin{lemma}\label{null}
Both 
\begin{itemize}
\item[(1)] \quad the multiplication by $y_{m+1}$ and 
\item[(2)] \quad the multiplication by  $e_k$ 
\end{itemize}
are trivial  on $\mathrm{Ker}\,\xi^*$.
\end{lemma}

\begin{proof}
Suppose that $z\in \mathrm{Ker}\, \xi^*$.
Then, $\xi^*(z \cdot y_{m+1})=0$ and $\deg (z \cdot y_{m+1})$ is odd. Hence,  we have $z \cdot y_{m+1}=0$ in $H^{*}(BG)$. 

We also get $Q_{k-1}(z \cdot y_{m+1})=0$. On the other hand,  since $\xi^*(Q_{k-1}z)=0$ and since $\deg (Q_{k-1} z)$ is odd, we have $Q_{k-1}z=0$ in $H^{*}(BG)$. Hence, we get 
\[
Q_{k-1}(z \cdot y_{m+1})=Q_{k-1}z \cdot y_{m+1}+(-1)^{k-1}z\cdot e_k=(-1)^{k-1}z\cdot e_k=0.
\]
So, we obtain $z\cdot e_k=0$. 
Thus, we have the desired results.
\end{proof}

By choosing suitable elements in $H^{*}(BG)$ as generators, we may consider a subalgebra of $H^
{*}(BG)$, which maps to $R$ in $H^{*}(BA)^{W(A)}$. By abuse of notation, we call it $R$.
Then, we may consider 
\[
E_{m+1}=\cdots=E_2=(M_0\oplus M_1 \oplus \mathrm{Ker}\, \xi^* )\otimes H^{*}(G/T),
\]
as an $R\otimes S$-module.
By Propositions~\ref{first-differential} and \ref{image} (2)  and Lemma~\ref{null} (1),  we have 
the $E_{m+2}$-term:
\[
E_{m+2}=(M_1 \otimes N_{p-1}) \oplus (M_1^{even}/(e_k) \otimes N_{\leq p-2}) \oplus (M_0\otimes N_0 )\oplus ( \mathrm{Ker}\,\xi^* \otimes H^{*}(G/T)), 
\]
where $N_{\leq i}$ is the $S$-submodule of $H^{*}(G/T)$ generated by $x_m^k$ ($k\leq i$) and 
$N_i$ is the $S$-submodule generated by a single element $x_m^i$ in $H^{*}(G/T)$. Observe that the above direct sum decomposition is  in the category of $R\otimes S$-modules.

Now, we investigate the action of the Weyl group on the spectral sequence  in terms of $\sigma_j$.
Recall that we define $\{ r_j\}$ to be the generators of the Weyl group $W$ and that $\sigma_j=1-r_j^*$
where  $\sigma_j$ acts on the spectral sequence by $\sigma_j(y \otimes x)=y \otimes \sigma_j (x)$ and so it commutes with the differential $d_r$ for $r\geq 2$.



\begin{lemma}\label{sigma}
There holds $\sigma_j (x_m^{i}) \in N_{\leq i-1}$ for all $\sigma_j$.
\end{lemma}

\begin{proof}
Since $d_{m+1}$ commutes with $\sigma_j$ and since $\sigma_j(y_{m+1}\otimes 1)=0$, we have 
\[d_{m+1}(\sigma_j(1\otimes x_m))=0.
\]
Suppose that $\sigma_j(x_m)=\beta x_m + s$ for some $\beta \in \mathbb{Z}/p$ and $s$ in $S$.
 Then, we have \[
 d_{m+1}(\beta (1 \otimes x_m)+ 1 \otimes s)=\alpha \beta (y_{m+1}\otimes 1)=0.
 \]
 Therefore, we have $\beta =0$ and $\sigma_j(x_m) \in N_0=S$.
In general, we have
\[
\sigma_j(x y)=\sigma_j(x)y+x\sigma_j(y)-\sigma_j(x)\sigma_j(y).
\]
Hence, we have 
\[
\sigma_j(x_m^i)=\sigma_j(x_m)x_m^{i-1}+x_m\sigma_j(x_m^{i-1})-\sigma_j(x_m)\sigma_j(x_m^{i-1})
\in N_{\leq i-1},\]
as desired.
\end{proof}



\begin{remark}\label{remark}
By Lemma~\ref{sigma}, $\sigma_j$ acts trivially on $N_{i}=N_{\leq i}/N_{\leq i-1}$. Hence, 
it is easy to see that 
\[
(E_{m+2}^{*, *'})^W=(M_1^{odd} \oplus e_k M_1^{even}) \otimes N_{p-1} \oplus (M_1^{even}/(e_k) \oplus M_0 \oplus \mathrm{Ker}\, \xi^*)\otimes \mathbb{Z}/p \not = E_{m+2}^{*, 0}.
\]
\end{remark}

Now, we begin to compute the next nontrivial differential. 



\begin{proposition}\label{trivial-differential}
For $r\geq m+2$ such that $E_r=E_{m+2}$, we have 
\[
d_r(M_0\otimes N_0)=d_r(\mathrm{Ker}\, \xi^* \otimes H^{*}(G/T))=d_r(M_1^{even}/(e_k) \otimes N_{\leq p-2})=\{0\}.
\]
\end{proposition}

\begin{proof}
Since $M_0\otimes N_0$ is generated by $M_0\otimes \mathbb{Z}/p$ as an $R\otimes S$-module, $d_r(M_0\otimes N_0)=\{0\}$ holds for $r\geq m+2$.
For $M_1^{even}/(e_k) \otimes N_{\leq p-2}$, there exist no odd degree generators. Hence, we have 
$d_r(M_1^{even}/(e_k) \otimes N_{\leq p-2})\subset E_{m+2}^{odd}=M_1^{odd}\otimes N_{p-1} \oplus M_0^{odd} \otimes N_0$.
On the one hand, the multiplication by $e_k\otimes 1$ is zero on $M_1^{even}/(e_k) \otimes N_{\leq p-2}$.
On the other hand, the multiplication by $e_k\otimes 1$ is a monomorphism on $M_1^{odd}\otimes N_{p-1} \oplus M_0^{odd} \otimes N_0$. Hence, we have 
\[
d_r(M_1^{even}/(e_k) \otimes N_{\leq p-2})=\{0\}.
\]
Finally, by Lemma~\ref{null}, the same holds for $\mathrm{Ker}\, \xi^{*} \otimes H^{*}(G/T)$ and so we obtain 
\[
d_r(\mathrm{Ker}\, \xi^{*} \otimes H^{*}(G/T))=\{0\}. \qedhere
\]
\end{proof}

Let $n=m(p-1)$ for the sake of notational simplicity. Next, we show the following proposition.



\begin{proposition}
If $r\geq m+2$ and if $d_r$ is nontrivial, then $r\geq n+1$.
\end{proposition}

\begin{proof}
Suppose that we have a nontrivial differential $d_r$ for some $r<n+1$, say,
\[
d_r(z \otimes x_m^{p-1})=z_{i_1} \otimes x'_1+\cdots +z_{i_\ell} \otimes x'_\ell,
\]
where $z \in M_1$, $1\leq i_1<\cdots < i_{\ell} \leq L$, $\{ z_1, \dots, z_L\}$ is a basis for 
\[
(M_1^{even}/(e_k) \oplus M_0 \oplus \mathrm{Ker}\, \xi^*)^{\deg z+r},
\]
and $x'_1, \dots, x'_\ell \in H^{n-r+1}(G/T)$, $x'_1, \dots, x'_\ell \not=0$.
Since $H^{*}(G/T)^W=\mathbb{Z}/p$, for $x'_1 \not =0$ in $H^{n-r+1}(G/T)$, there exists $\sigma_j$ such that
$\sigma_j(x'_1) \not=0$.  Therefore, we have 
\[
\sigma_j d_r(z \otimes x_{m}^{p-1})\not=0.
\]
On the other hand, by Lemma~\ref{sigma}, we have 
$\sigma_j(x_{m}^{p-1}) \in N_{\leq p-2}$. Hence, by Proposition~\ref{trivial-differential} above, we have
\[
\sigma_j d_r(z \otimes x_{m}^{p-1})\in d_r(M_1^{even}/(e_k) \otimes N_{\leq p-2})=\{0\}.
\]
This is a contradiction. Hence, we have $r\geq n+1$.
\end{proof}

Finally, we complete the computation of the spectral sequence.



\begin{proposition}\label{differential}
There holds $d_{n+1}(M_1\otimes N_{p-1})=(M_0^{odd}\oplus e_kM_0^{even}) \otimes N_0$.
\end{proposition}
\begin{proof}
The $E_{n+1}$-term is equal to 
\[
M_1\otimes N_{p-1} \oplus 
M_1^{even}/(e_k)\otimes N_{\leq p-2} \oplus M_0\otimes N_0 \oplus (\mathrm{Ker}\,\xi^*) \otimes H^{*}(G/T)
\]
and 
\[
d_{n+1}(M_1^{even}/(e_k)\otimes N_{\leq p-2} \oplus M_0\otimes N_0 \oplus (\mathrm{Ker}\,\xi^*) \otimes H^{*}(G/T))=\{0\}.
\]
Since $M_1^{even}/(e_k)\otimes N_{\leq p-2} \oplus M_0\otimes N_0 \oplus (\mathrm{Ker}\,\xi^*) \otimes H^{*}(G/T)$ is generated by the elements of the second degree less than $n$, that is, the elements in $E_{r}^{*,*'}$ ($*'<n$),  it is clear that 
\[
d_r(M_1^{even}/(e_k)\otimes N_{\leq p-2} \oplus M_0\otimes N_0 \oplus (\mathrm{Ker}\,\xi^*) \otimes H^{*}(G/T))=\{0\}
\]
for all $r\geq n+1$.

On the other hand, since all the elements in $(M_0^{odd}\oplus e_kM_0^{even})\otimes \mathbb{Z}/p$ do not survive to the $E_\infty$-term and since $d_r(M_0\otimes N_0)=\{0\}$ for all $r\geq 2$, 
all the elements in $(M_0^{odd}\oplus e_kM_0^{even})\otimes \mathbb{Z}/p$ must be hit by nontrivial differentials.

Suppose that there exists an element in $(M_0^{odd}\oplus e_kM_0^{even})\otimes \mathbb{Z}/p$ which is not hit by $d_{n+1}$.
Let $z \otimes 1$ be such an element with the lowest degree $s$.
Up to degree $<s$, by Proposition~\ref{image}, we see that 
\[
d_{n+1}:M_{1}^{i} \otimes N_{p-1} \to (M_0^{odd} \oplus e_kM_0^{even})^{i+n+1} \otimes N_0
\]
is an isomorphism for $i<s$.

Then,  $\mathrm{Ker}\, d_{n+1}$
is equal to $M_1^{even}/(e_k)\otimes N_{\leq p-2} \oplus M_0\otimes N_0 \oplus (\mathrm{Ker}\,\xi^*) \otimes H^{*}(G/T)$ up to degree $s$.
Therefore, for $r\geq n+2$, $\mathrm{Im}\, d_{r} =\{0\}$ up to degree $\leq s$. Hence the element $z \otimes 1$ survives to the $E_{\infty}$-term. This is a contradiction. 
So, the proposition holds.
\end{proof}

Thus, by Propositions \ref{trivial-differential} and \ref{differential}, we have 
\[
E_{n+2}=(M_1^{even}/(e_k) \otimes N_{\leq p-2})  \oplus (M_0^{even}/(e_k) \otimes N_0) \oplus
( \mathrm{Ker}\, \xi^* \otimes H^{*}(G/T)).
\]
Since there are no odd degree elements in the $E_{n+2}$-term, the spectral sequence collapses at the $E_{n+2}$-level and we obtain $E_\infty=E_{n+2}$ 
and \[
(E_{\infty}^{*,*'})^W=E_{\infty}^{*,0}=(M_1^{even}/(e_k)  \oplus M_0^{even}/(e_k) \oplus \mathrm{Ker}\, \xi^* )\otimes \mathbb{Z}/p.
\]
This completes the proof of Theorem~\ref{main}.



\end{document}